\documentclass{amsart}

\usepackage{amsmath,amscd}
\usepackage{amssymb}

\newtheorem{theorem}{Theorem}[section]
\newtheorem{lemma}[theorem]{Lemma}
\newtheorem{proposition}[theorem]{Proposition}
\newtheorem{corollary}[theorem]{Corollary}
\newtheorem{conjecture}[theorem]{Conjecture}

\theoremstyle{definition}

\theoremstyle{remark}

\numberwithin{equation}{section}

\begin{document}

\title{Fourier Transforms of Nilpotent Coadjoint Orbits for $\text{GL}(n,\mathbb{R})$}

\author{Benjamin Harris}
\address{Department of Mathematics, Louisiana State University, Baton Rouge, LA 70803}
\email{blharris@lsu.edu}
\thanks{The author was supported in part by an NSF Graduate Research Fellowship. He was a graduate student at the Massachusetts Institute of Technology when the research was conducted.}



\date{December 3, 2011.}



\begin{abstract}
The main result of this paper is an explicit formula for the Fourier transform of the canonical measure on a nilpotent coadjoint orbit for $\text{GL}(n,\mathbb{R})$. This paper also includes some results on limit formulas for reductive Lie groups including new proofs of classical limit formulas of Rao and Harish-Chandra.
\end{abstract}

\maketitle

\section{Introduction}
Let $G$ be a reductive Lie group, and let $\pi$ be an irreducible, admissible representation of $G$ with character $\Theta_{\pi}$. By results of \cite{BV} and \cite{SV}, the leading term of $\Theta_{\pi}$ at one is an integral linear combination of Fourier transforms of canonical measures on nilpotent coadjoint orbits. Further, every orbit which occurs in this sum has the same complexification. In particular, if $G=\text{GL}(n,\mathbb{R})$, then the leading term of any irreducible character of $G$ is a positive integer times the Fourier transform of the canonical measure on a nilpotent coadjoint orbit.\\
\indent The main result of this paper is an explicit formula for the Fourier transform of the canonical measure on a nilpotent coadjoint orbit for $\text{GL}(n,\mathbb{R})$. Given a conjugacy class of Levi subgroups, $\mathcal{L}$, for $\text{GL}(n,\mathbb{R})$, fix a conjugacy class of parabolics $\mathcal{P}$ with Levi factor $\mathcal{L}$. Then we define
$$\mathcal{O}_{\mathcal{L}}\subset \mu(T^*\mathcal{P})$$
to be the unique open orbit. Here $\mu$ is the moment map (defined in section four). In the next statement, we also denote the canonical measure on this orbit (defined in section two) by $\mathcal{O}_{\mathcal{L}}$. Moreover, whenever $G$ is a reductive Lie group and $H$ is a Cartan subgroup, then $W(G,H)=N_G(H)/H$ denotes the real Weyl group of $G$ with respect to $H$. Here is our main result.

\begin{theorem}
Fix a Cartan subgroup $H\subset G=\text{GL}(n,\mathbb{R})$, let $\mathfrak{h}=\text{Lie}(H)$, and let $C$ be a connected component of the regular set $\mathfrak{h}'$ of $\mathfrak{h}$. Choose positive roots $\Delta^+$ of $\mathfrak{g}_{\mathbb{C}}$ with respect to $\mathfrak{h}_{\mathbb{C}}$ satisfying:\\
(i) If $\alpha$ is a positive real root and $X\in C$, then $\alpha(X)>0$.\\
(ii) If $\alpha$ is a complex root, then $\alpha$ is positive iff $\overline{\alpha}$ is positive.\\
Then
$$\widehat{\mathcal{O}_{\mathcal{L}}}|_C=\sum_{L\in \mathcal{L},\ L\supset H} \left|\frac{W(G,H)_L}{W(L,H)}\right|\frac{\pi_L}{\pi}.$$
Here $\pi=\prod_{\alpha\in \Delta^+}\alpha$, $\Delta_L^+$ denotes the roots of $L$ that lie in $\Delta^+$, $\pi_L=\prod_{\alpha\in \Delta_L^+}\alpha$, and $W(G,H)_L=\{w\in W(G,H)|\ wL=L\}$.
\end{theorem}

In the process of proving this result, we will write down a number of limit formulas for reductive Lie groups. First, we have a limit formula for semisimple orbits.\\
\indent If $G$ is a reductive Lie group, we will write $r(G)$ (or simply $r$) for one half the number of roots of $G$ with respect to any Cartan $H$. If we fix a Cartan $H$, then $q(G,H)$ will denote one half the number of non-compact, imaginary roots of $G$ with respect to $H$. If $H$ is a fundamental Cartan, then we will write $q(G)$ (or simply $q$) instead of $q(G,H)$. 

\begin{theorem} [Harish-Chandra] Let $G$ be a reductive Lie group, let $\xi\in \mathfrak{g}^*=\text{Lie}(G)^*$ be a semisimple element, and let $L=Z_G(\xi)\subset G$ be the corresponding reductive subgroup. Choose positive roots $\Delta_L^+\subset \Delta_L=\Delta(\mathfrak{l}_{\mathbb{C}},\mathfrak{h}_{\mathbb{C}})$ such that a complex root $\alpha$ is positive iff $\overline{\alpha}$ is positive, and put 
$$C=\{\lambda\in (\mathfrak{h}^*)'|\ \langle i\lambda, \alpha^{\vee}\rangle>0\ \text{for\ all\ imaginary\ roots}\ \alpha\in \Delta_L^+\}.$$
Then 
$$\lim_{\lambda\rightarrow \xi,\ \lambda\in C} \partial(\pi_L)|_{\lambda} \mathcal{O}_{\lambda}=i^{r(L)}(-1)^{q(L)}|W(L,H)|\mathcal{O}_{\xi}.$$
\noindent Here $\pi_L$ is the product of the positive roots of $L$, and $\mathfrak{h}^*=\text{Lie}(H)^*\subset \mathfrak{g}^*=\text{Lie}(G)^*$ is embedded in the usual way.
\end{theorem}

We give Harish-Chandra credit for this result because he proved a group analogue of this theorem on pages 33-34 in \cite{HC6}. The special case where $\xi=0$ was proved by Harish-Chandra in \cite{HC1}, \cite{HC2}, \cite{HC3}, and \cite{HC7}. In \cite{Bo1}, Bocizevic gives a proof of the case $\xi=0$ using techniques developed by Schmid and Vilonen. In sections one and two, we write down a proof of this theorem using elementary, classical methods. This simple proof is probably well-known to experts. However, we wish to write it down since it does not appear in the literature.\\
\indent Next, we have limit formulas for nilpotent orbits. Let $\nu\in \mathfrak{g}^*$, and let $\mathcal{O}_{\nu}$ denote the canonical measure on the coadjoint orbit $G\cdot \nu$. Then the limit of distributions
$$\lim_{t\rightarrow 0^+} \mathcal{O}_{t\nu}=\sum n_G(\mathcal{O},\nu)\mathcal{O}$$
is a sum of canonical measures on nilpotent coadjoint orbits. Let $\mathcal{O}_{\mathbb{C}}$ denote the $\text{Int}\mathfrak{g}_{\mathbb{C}}$-orbit $\text{Int}\mathfrak{g}_{\mathbb{C}}\cdot \mathcal{O}\subset \mathfrak{g}_{\mathbb{C}}^*$. 

\begin{proposition} The coefficients $n_G(\mathcal{O},\nu)$ are positive integers. Moreover, we have the inequality $n_G(\mathcal{O},\nu)\leq n_{\operatorname{Int}\mathfrak{g}_{\mathbb{C}}}(\mathcal{O}_{\mathbb{C}},\nu)$.  
\end{proposition}
 
In \cite{Ba}, Dan Barbasch gives formulas for $n_{\operatorname{Int}\mathfrak{g}_{\mathbb{C}}}(\mathcal{O}_{\mathbb{C}},\nu)$. Hence, the above theorem gives an upper bound for $n_G(\mathcal{O},\nu)$. In section 4, we give a lower bound for $n_G(\mathcal{O},\nu)$ in the case where $n_G(\mathcal{O},\nu)\neq 0$. When the lower and upper bounds coincide, we get a formula for $n_G(\mathcal{O},\nu)$. We will show that this happens for certain $\text{GL}(n,\mathbb{R})$ limit formulas, and we use these formulas together with Theorem 1.2 to explicitly compute the formulas in Theorem 1.1. These bounds also coincide when $\mathcal{O}_{\mathbb{C}}$ is an even orbit. Rao proved but never published a limit formula for even nilpotent orbits. Recently in \cite{Bo2}, Bozicevic gave a deep, modern proof of Rao's result. In section five, we use the above results to give an elementary, classical proof of Rao's limit formula.

\section{Harish-Chandra's Limit Formula for the Zero Orbit}

In this section, we prove Harish-Chandra's limit formula for the zero orbit. It was proven in \cite{HC1}, \cite{HC2}, \cite{HC3}, and \cite{HC7}, but with a different normalization of the measures on coadjoint orbits than the one we use here. In \cite{Bo1}, using methods of Schmid and Vilonen, Bozicevic gave a proof of the formula, written in terms of canonical measures on orbits. In this section, we show how to write down such a proof using only well-known, classical methods. This proof may be well-known to experts, but it has not appeared in the literature. Moreover, it is not a waste of space to write it down here since many of the fundamental results we recall in this section will be needed later in the article for other purposes.\\
\indent First, we need a couple of definitions. A Lie group $G$ is \emph{reductive} if there exists a real, reductive algebraic group $G_1$ and a Lie group homomorphism $G\rightarrow G_1$ with open image and finite kernel. Let $\mathcal{O}\subset \mathfrak{g}^*$ be a coadjoint orbit for $G$. The Kostant-Kirillov symplectic form $\omega$ is defined on $\mathcal{O}$ by the formula
$$\omega_{\lambda}(\text{ad}_X^*\lambda,\text{ad}_Y^*\lambda)=\lambda([X,Y]).$$
The top dimensional form
$$\frac{\omega^m}{m! (2\pi)^m}$$
on $\mathcal{O}$ gives rise to the canonical measure on $\mathcal{O}$. Here $m=\frac{\dim \mathcal{O}}{2}$. We will often abuse notation and write $\mathcal{O}$ for the orbit as well as the canonical measure on the orbit. In what follows, we will denote the $G$-orbit through $\lambda$ by $\mathcal{O}_{\lambda}^G$ (or sometimes just $\mathcal{O}_{\lambda}$). 

\begin{theorem} [Harish-Chandra] Let $G$ be a reductive Lie group, and let $H\subset G$ be a fundamental Cartan subgroup. Choose positive roots $\Delta^+\subset \Delta=\Delta(\mathfrak{g}_{\mathbb{C}},\mathfrak{h}_{\mathbb{C}})$ such that a complex root $\alpha$ is positive iff $\overline{\alpha}$ is positive, and put 
$$C=\{\lambda\in (\mathfrak{h}^*)'|\ \langle i\lambda, \alpha^{\vee}\rangle>0\ \text{for\ all\ imaginary\ roots}\ \alpha\in \Delta^+\}.$$
Then 
$$\lim_{\lambda\rightarrow 0,\ \lambda\in C} \partial(\pi)|_{\lambda} \mathcal{O}_{\lambda}=i^{r}(-1)^{q}|W(G,H)|\delta_0.$$
\noindent Here $\pi$ is the product of the positive roots of $G$ and $\mathfrak{h}^*=\text{Lie}(H)^*\subset \mathfrak{g}^*=\text{Lie}(G)^*$ is embedded in the usual way.
\end{theorem}

We will actually prove the Fourier transform of the above theorem. Recall the definition of the Fourier transform. Let $V$ be a finite dimensional, real vector space, and let $\mu$ be a smooth, rapidly decreasing measure (that is, a Schwartz function multiplied by a Lebesgue measure) on $V$ . Then the Fourier transform of $\mu$ is defined to be
$$\widehat{\mu}(l)=\int_V e^{i\langle l,X\rangle}d\mu(X).$$
Note $\widehat{\mu}$ is a Schwartz function on $V^*$. Given a tempered distribution $D$ on $V^*$, its Fourier transform $\widehat{D}$ is a tempered, generalized function on $V$ defined by 
$$\langle \widehat{D}, \mu\rangle:=\langle D,\widehat{\mu}\rangle.$$

Next, we recall Harish-Chandra's result on Fourier transforms of regular, semisimple orbits. If $\mathfrak{h}\subset \mathfrak{g}$ is a Cartan subalgebra, define $$\mathfrak{h}''=\{X\in \mathfrak{h}|\ \alpha(X)\neq 0\ \forall \alpha\in \Delta_{\text{real}}\}.$$ Here $\Delta_{\text{real}}$ denotes the real roots of $\mathfrak{g}$ with respect to $\mathfrak{h}$. 
\begin{lemma} [Harish-Chandra] Let $\mathfrak{h}\subset \mathfrak{g}$ be a Cartan, let $C\subset (\mathfrak{h}^*)'\subset \mathfrak{g}^*$ be a connected component of the set of regular elements in $\mathfrak{h}^*$, and let $\lambda\in C$ be regular, semisimple. Suppose $\mathfrak{h}_1\subset \mathfrak{g}$ is a Cartan subalgebra, and $C_1\subset \mathfrak{h}_1''$ is a connected component. Identify $\lambda\in \mathfrak{h}_{\mathbb{C}}^*\cong (\mathfrak{h}_1)_{\mathbb{C}}^*$ via an inner automorphism of $\text{Int}\mathfrak{g}_{\mathbb{C}}$. Then
$$\widehat{\mathcal{O}_{\lambda}}|_{C_1}=\frac{\sum_{w\in W_{\mathbb{C}}} a_w e^{iw\lambda}}{\pi}$$
for $a_w\in \mathbb{C}$ constants. Here $W_{\mathbb{C}}$ is the Weyl group of the roots of $\mathfrak{g}_{\mathbb{C}}$ with respect to $(\mathfrak{h}_1)_{\mathbb{C}}$, and $\pi$ is a product of positive roots of $\mathfrak{g}_{\mathbb{C}}$ with respect to $(\mathfrak{h}_1)_{\mathbb{C}}$. Further, the constants $a_w$ are independent of the choice of $\lambda\in C$.
\end{lemma}

This is essentially Lemma 24 of \cite{HC1}. Differentiating the above formula with respect to $\lambda$ yields
$$\lim_{\lambda\in C,\ \lambda\rightarrow 0}\partial(\pi)|_{\lambda}\widehat{\mathcal{O}_{\lambda}}|_{C_1}=i^{r}\left(\sum \epsilon(w)a_w\right)\frac{\pi}{\pi}.$$
Observe that the coefficients $a_w$ depend on a component $C\subset (\mathfrak{h}^*)'$ as well as a component $C_1\subset (\mathfrak{h}_1)''$. For the remainder of the section, we fix $C\subset (\mathfrak{h}^*)'$ and we assume that $\mathfrak{h}$ is a fundamental Cartan subalgebra. 
To prove Theorem 2.1, we need only show the following lemma.

\begin{lemma} Assume $\mathfrak{h}$ is a fundamental Cartan. Then $$\sum \epsilon(w) a_w=(-1)^{q}|W(G,H)|.$$
\end{lemma}

Again, these coefficients $a_w$ depend on a component $C_1\subset (\mathfrak{h}_1)''$. In order to prove Lemma 2.3, we first prove it in the case where $\mathfrak{h}_1=\mathfrak{h}$ using a result of Rossmann and Harish-Chandra descent. For our applications, it is important to give Berline-Vergne's formulation \cite{BeVe} of Rossmann's result \cite{R1} (we recommend the proof of Berline-Vergne as well).

\begin{theorem} [Rossmann] Let $G$ be a reductive Lie group, and let $H$ be a Cartan subgroup such that $H/Z(G)$ is compact. Let $\lambda\in C\subset (\mathfrak{h}^*)'=(\text{Lie}(H)^*)'$ be regular, semisimple, and choose positive roots $\Delta^+\subset \Delta$ satisfying
$\langle i\lambda, \alpha^{\vee}\rangle>0$ for all $\alpha^{\vee}\in (\Delta^+)^{\vee}$. Then $$\widehat{\mathcal{O}_{\lambda}}|_{\mathfrak{h}'}=(-1)^{q}\frac{\sum_{w\in W(G,H)}\epsilon(w)e^{iw\lambda}}{\pi}$$
where $\pi$ is the product of the positive roots.
\end{theorem}

\indent This theorem implies Lemma 2.3 when $\mathfrak{h}_1=\mathfrak{h}$ when $G$ is of equal rank. Now, let $G$ be an arbitrary reductive Lie group, and let $\mathfrak{h}\subset \mathfrak{g}=\text{Lie}(G)$ be a fundamental Cartan. Choose a Cartan involution $\theta$ such that $\mathfrak{h}$ is $\theta$ stable with decomposition $\mathfrak{h}=\mathfrak{t}\oplus \mathfrak{a}$. Then $M=Z_G(\mathfrak{a})$ is a reductive Lie group of equal rank. Harish-Chandra gave continuous maps \cite{HC1}
$$\phi:\ \mathcal{S}(\mathfrak{g}^*)\rightarrow \mathcal{S}(\mathfrak{m}^*),\ \psi:\ \mathcal{SM}(\mathfrak{g})\rightarrow \mathcal{SM}(\mathfrak{m})$$ well-defined up to a constant, where $\mathcal{S}(V)$ is the space of smooth, rapidly decreasing functions on a vector space $V$ and $\mathcal{SM}(V)$ is the space of smooth, rapidly decreasing measures on a vector space $V$. Dualizing, we obtain maps
$$\phi^*:\ \text{TD}(\mathfrak{m}^*)\rightarrow \text{TD}(\mathfrak{g}^*),\ \psi^*=\text{HC}:\ \text{TGF}(\mathfrak{m})\rightarrow \text{TGF}(\mathfrak{g})$$
on tempered distributions and tempered generalized functions. We call the map on the right $\text{HC}$ because it is Harish-Chandra's descent map. Thus far, these maps are only well-defined up to a constant; however, there is a nice way to normalize this constant. It follows from results of Rossmann \cite{R2} that one can fix the constant on $\phi^*$ so that $\phi^*(\mathcal{O}_{\lambda}^M)=\mathcal{O}_{\lambda}^G$ takes canonical measures on $G$-regular, semisimple coadjoint orbits to canonical measures on regular, semisimple coadjoint orbits. In \cite{HC1}, Harish-Chandra observes that $\psi$ is (up to a constant) the Fourier transform of $\phi$. Thus, we may require $$\text{HC}(\widehat{D})=\widehat{\psi^*(D)}$$
for all $D\in\text{TD}(\mathfrak{m}^*)$ and this precisely defines the map $\text{HC}$.

\indent We fix this normalization. Arguments similar to the ones in \cite{R2} imply the following explicit formula for computing $\text{HC}$.

\begin{lemma} [Harish-Chandra, Rossmann] Let $F$ be an $M$-invariant generalized function on $\mathfrak{m}$ that is given by integration against an analytic function on the set of regular, semisimple elements $\mathfrak{m}'\subset \mathfrak{m}$, which we also denote by $F$. Given $X\in \mathfrak{g}'$, let $\{Y_i\}_{i=1}^k$ be a set of representatives for the finite number of $M$-orbits in $\mathcal{O}^G_X\cap \mathfrak{m}$. Then $\text{HC}(F)$ is a $G$-invariant generalized function on $\mathfrak{g}$ that is given by integration against an analytic function on the set of regular, semisimple elements $\mathfrak{g}'\subset \mathfrak{g}$, which we also denote by $\text{HC}(F)$. Explicitly, we have
$$\text{HC}(F)(X)=\sum_{i=1}^k F(Y_i)\left| \pi_{G/M}(Y_i) \right|^{-1}.$$
To define $|\pi_{G/M}(Y)|$, choose a Cartan $Y\in \mathfrak{h}\subset \mathfrak{m}$, let $\Delta_G$ (resp. $\Delta_M$) be the roots of $\mathfrak{g}$ (resp. $\mathfrak{m}$) with respect to $\mathfrak{h}$, let $\Delta_G^+$ be a choice of positive roots of $\Delta_G$, and let $\Delta_M^+=\Delta^+_G\cap \Delta_M$. Then $|\pi_{G/M}(Y)|=|\prod_{\alpha\in \Delta^+_G\setminus \Delta^+_M} \alpha(Y)|$. This definition is independent of the above choices.
\end{lemma}

\noindent Combining Theorem 2.4 and Lemma 2.5, we get the following corollary.

\begin{corollary} [Rossmann] Let $G$ be a reductive Lie group with Cartan subgroup $H$, and let $q(G,H)$ be half the number of non-compact imaginary roots of $G$ with respect to $H$. Let $\lambda\in \mathfrak{h}^*=\text{Lie}(H)^*$ be a regular element, and let $C_1\subset \mathfrak{h}'$ be a connected component. Choose positive roots $\Delta^+\subset \Delta$ satisfying\\
\noindent (i) If $\alpha^{\vee}\in (\Delta^+)^{\vee}_{\text{imag.}}$, then $\langle i\lambda, \alpha^{\vee}\rangle>0$.\\
\noindent (ii) If $\alpha\in \Delta^+_{\text{real}}$ and $X\in C_1$, then $\alpha(X)>0$.\\
\noindent (iii) If $\alpha$ a complex root, then $\alpha\in \Delta^+$ iff $\overline{\alpha}\in \Delta^+$.

\noindent Then $$\widehat{\mathcal{O}_{\lambda}^G}|_{C_1}=(-1)^{q(G,H)}\frac{\sum_{w\in W(G,H)}\epsilon_I(w)e^{iw\lambda}}{\pi}$$
where $W(G,H)=N_G(H)/H$, $\pi$ is the product of the positive roots, and $\epsilon_I$ is defined by $$w\cdot \pi_I=\epsilon_I(w)\pi_I,\ \pi_I=\prod_{\alpha\in \Delta^+_{\text{imag}}}\alpha.$$
\noindent Moreover, $\widehat{\mathcal{O}_{\lambda}^G}$ is zero on Cartan subalgebras $\mathfrak{h}$ which are not conjugate to a Cartan subalgebra of $Z_{\mathfrak{g}}(\lambda)$.
\end{corollary}

\noindent A version of this result containing a few typos can be found in \cite{R3}. Since $\epsilon_I=\epsilon$ on a fundamental Cartan $\mathfrak{h}$, this verifies Lemma 2.3 on $\mathfrak{h}'$. To finish the proof of Lemma 2.3, we use Harish-Chandra's matching conditions.

\begin{theorem} [Harish-Chandra] Let $\mathfrak{h}\subset \mathfrak{g}$ be a fundamental Cartan subalgebra, and let $\mathfrak{h}_1$ be another Cartan subalgebra. Let $\lambda\in (\mathfrak{h}^*)'$, and identify $\lambda\in \mathfrak{h}_{\mathbb{C}}^*\cong (\mathfrak{h}_1)_{\mathbb{C}}^*$ via an inner automorphism of $\mathfrak{g}_{\mathbb{C}}$. Suppose $\mathfrak{h}_2$ is a third Cartan related to $\mathfrak{h}_1$ by a Cayley transform $c_{\alpha}$ via a noncompact, imaginary root $\alpha$ of $\mathfrak{h}_1$. Let $C_1\subset \mathfrak{h}_1''$ be a connected component containing an open subset of $\text{ker}(\alpha)$, and let $C_2\subset \mathfrak{h}_2''$ be a connected component such that $C_1$ contains a wall of $\overline{C_2}$. Suppose
$$\widehat{\mathcal{O}_{\lambda}}|_{C_1}=\frac{\sum a_w e^{iw\lambda}}{\pi},\ \ \widehat{\mathcal{O}_{\lambda}}|_{C_2}=\frac{\sum b_w e^{iw\lambda}}{\pi}.$$
Then we have $$\epsilon(w)a_w+\epsilon(s_{\alpha}w)a_{s_{\alpha}w}=\epsilon(w)b_w+\epsilon(s_{\alpha}w)b_{s_{\alpha}w}.$$
Here we identify the noncompact, imaginary root $\alpha$ of $\mathfrak{h}_1$ with the corresponding real root of $\mathfrak{h}_2$. We also identify $\lambda\in (\mathfrak{h}_1)_{\mathbb{C}}^*$ with the corresponding element $\lambda\in (\mathfrak{h}_2)_{\mathbb{C}}^*$ via the pullback of the isomorphism $c_{\alpha}:(\mathfrak{h}_1)_{\mathbb{C}}\rightarrow (\mathfrak{h}_2)_{\mathbb{C}}$.
\end{theorem}

This theorem is Lemma 26 of \cite{HC5} where Harish-Chandra remarks that it follows from Lemma 18 of \cite{HC4}. Summing these relations over the entire Weyl group, we get
$$\sum \epsilon(w) a_w=\sum \epsilon(w) b_w.$$
Since any component of any Cartan can be related to a component of a fundamental Cartan via successive Cayley transforms, we deduce 
$$\sum \epsilon(w) a_w=(-1)^{q}|W(G,H)|$$
whenever $\widehat{\mathcal{O}_{\lambda}}|_{C_1}=\frac{\sum a_we^{w\lambda}}{\pi}$ on any component $C_1\subset \mathfrak{h}_1''$ for any Cartan $\mathfrak{h}_1$. This is the statement of Lemma 2.3. As we have already remarked, Theorem 2.1 follows.

\section{Harish-Chandra's Limit Formula for Semisimple Orbits}

In this section, we prove Harish-Chandra's limit formula for an arbitrary semisimple orbit. A group analogue of this result was proved on pages 33-34 of \cite{HC6}.

\begin{theorem} Let $G$ be a reductive Lie group, let $\xi\in \mathfrak{g}^*=\text{Lie}(G)^*$ be a semisimple element, let $L=Z_G(\xi)\subset G$ be the corresponding reductive subgroup, and fix a fundamental Cartan subgroup $H\subset L$. Choose positive roots $\Delta_L^+\subset \Delta_L=\Delta(\mathfrak{l}_{\mathbb{C}},\mathfrak{h}_{\mathbb{C}})$ such that a complex root $\alpha$ is positive iff $\overline{\alpha}$ is positive, and put 
$$C=\{\lambda\in (\mathfrak{h}^*)'|\ \langle i\lambda, \alpha\rangle>0\ \text{for\ all\ imaginary\ roots}\ \alpha\in \Delta^+\}.$$
Then 
$$\lim_{\lambda\rightarrow \xi,\ \lambda\in C} \partial(\pi_L)|_{\lambda} \mathcal{O}_{\lambda}=i^{r(G)}(-1)^{q(L)}|W(L,H)|\mathcal{O}_{\xi}.$$
\noindent Here $\pi_L$ is the product of the positive roots of $L$ and $\mathfrak{h}^*=\text{Lie}(H)^*\subset \mathfrak{g}^*=\text{Lie}(G)^*$ is embedded in the usual way.
\end{theorem}

We prove the theorem by reducing to the case $\xi=0$, which was proved in the last section. Let $d_{G/H}$ be a Haar measure on $G/H$, let $d_{G/L}$ be a Haar measure on $G/L$, and let $d_{L/H}$ be a Haar measure on $L/H$. Then it is a well-known fact (see for instance page 95 of \cite{H}) that there exists a constant $c>0$ such that 
$$\int_{G/H}f(g\cdot X)d_{G/H}g=c\int_{G/L} \left(\int_{L/H} f(gl\cdot X) d_{L/H}l\right)d_{G/L}g.$$
Choosing a Haar measure on $G/H$ (resp. $G/L$, $L/H$) is equivalent to choosing a top-dimensional, alternating tensor $\eta_{G/H}$, well-defined up to sign, on $(\mathfrak{g}/\mathfrak{h})^*$ (resp. $\eta_{G/L}$, $\eta_{L/H}$ on $(\mathfrak{g}/\mathfrak{l})^*$, $(\mathfrak{l}/\mathfrak{h})^*$). The exact sequence
$$0\rightarrow (\mathfrak{g}/\mathfrak{l})^*\rightarrow (\mathfrak{g}/\mathfrak{h})^*\rightarrow (\mathfrak{l}/\mathfrak{h})^*\rightarrow 0$$
gives rise to maps on alternating tensors. Abusing notation, we also write $\eta_{G/L}$ for the image of $\eta_{G/L}$ under the above map, and we also write $\eta_{L/H}$ for a preimage of $\eta_{L/H}$ under the above map. Then
$$\eta_{G/H}=\pm c(\eta_{G/L}\wedge \eta_{L/H}).$$
This can be proved by relating the multiplication $gl$ on the group to addition on the Lie algebra and then applying Fubini's theorem.\\
\indent To apply these remarks to our proof of the theorem, fix a Haar measure on the homogeneous space $G/L$ by identifying $G/L\cong \mathcal{O}_{\xi}$ and using the canonical measure, fix a Haar measure on $G/H$ by identifying $G/H\cong \mathcal{O}_{\lambda}$ for a fixed $\lambda\in C$, and fix a Haar measure on $L/H$ by identifying $L/H\cong \mathcal{O}_{\lambda}^L$. Then we get
$$\int_{G/H} f(g\cdot \lambda)dg=c_{\lambda}\int_{G/L}\left(\int_{L/H} f(gl\cdot \lambda) dl\right)dg.$$

\begin{lemma} Let $\Delta_G$ (resp. $\Delta_L$) denote the roots of $\mathfrak{g}$ (resp. $\mathfrak{l}$) with respect to $\mathfrak{h}$. Let $\Delta_G^+\subset \Delta_G$ be a choice of positive roots, and let $\Delta_L^+=\Delta_G^+\cap \Delta_L$. Then 
$$c_{\lambda}=\frac{\prod_{\alpha\in \Delta_G^+\setminus \Delta_L^+}\langle \lambda, \alpha^{\vee}\rangle}{\prod_{\alpha\in \Delta_G^+\setminus \Delta_L^+}\langle \xi, \alpha^{\vee}\rangle}$$
for $\lambda\in C$. In particular, $\displaystyle \lim_{\lambda\in C,\ \lambda\rightarrow \xi} c_{\lambda}=1$.
\end{lemma}

\begin{proof} Recall that the forms $\eta_{G/H}$, $\eta_{G/L}$, and $\eta_{L/H}$ yield top-dimensional alternating tensors on $(\mathfrak{g}/\mathfrak{h})^*$, $(\mathfrak{g}/\mathfrak{l})^*$, and $(\mathfrak{l}/\mathfrak{h})^*$, well-defined up to a choice of sign. Extend these tensors complex linearly to $(\mathfrak{g}_{\mathbb{C}}/\mathfrak{h}_{\mathbb{C}})^*$, $(\mathfrak{g}_{\mathbb{C}}/\mathfrak{l}_{\mathbb{C}})^*$, $(\mathfrak{l}_{\mathbb{C}}/\mathfrak{h}_{\mathbb{C}})^*$ and denote them by $\eta_{\lambda}^G$, $\eta_{\xi}^G$, and $\eta_{\lambda}^L$. Note that we still have the identity
$\eta_{\lambda}^G=\pm c_{\lambda}(\eta_{\xi}^G\wedge \eta_{\lambda}^L)$.
Consider the root space decomposition 
$$\mathfrak{g}_{\mathbb{C}}=\mathfrak{h}_{\mathbb{C}}\bigoplus \left(\sum_{\alpha\in \Delta_G} (\mathfrak{g}_{\mathbb{C}})_{\alpha}\right).$$
For each $\alpha\in \Delta_G^+$, choose elements $X_{\alpha}\in (\mathfrak{g}_{\mathbb{C}})_{\alpha}$, $X_{-\alpha}\in (\mathfrak{g}_{\mathbb{C}})_{-\alpha}$, and $H_{\alpha}\in \mathfrak{h}_{\mathbb{C}}$ such that $\{X_{\alpha},H_{\alpha},X_{-\alpha}\}$ is an $\mathfrak{sl}_2$-triple. Then
$$\left[(2\pi)^m\eta_{\lambda}^G(\{X_{\alpha}\}_{\alpha\in\Delta_G})\right]^2=\left(\frac{\omega_{\lambda}^m}{m!}\left(\{X_{\alpha}\}_{\alpha\in \Delta_G}\right)\right)^2=\det\left(\omega_{\lambda}(X_{\alpha},X_{\beta})\right).$$
Here $m=\frac{1}{2}(\dim \mathfrak{g}-\dim \mathfrak{h})$ and $\omega_{\lambda}$ is the Kostant-Kirillov symplectic form on $\mathcal{O}_{\lambda}^G$. Note that we need not order the tangent vectors $\{X_{\alpha}\}$ before applying the square of the top dimensional alternating tensor $\eta_{\lambda}^G$ to them since the value of $\eta_{\lambda}^G(\{X_{\alpha}\}_{\alpha\in \Delta_G})^2$ is independent of this ordering. The second equality follows from explicitly expanding out $\omega_{\lambda}^m(\{X_{\alpha}\})$ into a sum with signs, squaring it, and identifying the result as the corresponding $2m$ by $2m$ determinant. Finally, we have
$$\det(\omega_{\lambda}(X_{\alpha},X_{\beta})=\prod_{\alpha\in \Delta^+_G} \lambda(H_{\alpha})^2.$$
This follows from the fact that $\omega_{\lambda}(X_{\alpha},X_{\beta})=\lambda[X_{\alpha},X_{\beta}]\neq 0$ only if $\beta=-\alpha$ in which case we obtain $\lambda(H_{\alpha})$. If $k=\frac{1}{2}(\dim \mathfrak{g}-\dim \mathfrak{l})$ and $l=\frac{1}{2}(\dim \mathfrak{l}-\dim \mathfrak{h})$, then we similarly have
$$\left[(2\pi)^k\eta_{\xi}^G(\{X_{\alpha}\}_{\alpha\in \Delta_G\setminus \Delta_L})\right]^2=\prod_{\alpha\in \Delta^+_G\setminus \Delta^+_L}\xi(H_{\alpha})^2$$
and
$$\left[(2\pi)^l\eta_{\lambda}^G(\{X_{\alpha}\}_{\alpha\in \Delta_L})\right]^2=\prod_{\alpha\in \Delta^+_L}\lambda(H_{\alpha})^2.$$
Combining the above formulas and the identity $\eta_{\lambda}^G=\pm c_{\lambda}(\eta_{\xi}^G\wedge \eta_{\lambda}^L)$ yields
$$c_{\lambda}^2\prod_{\alpha\in \Delta^+_G\setminus \Delta^+_L} \xi(H_{\alpha})^2\prod_{\alpha\in \Delta_L^+}\lambda(H_{\alpha})^2=\prod_{\alpha\in \Delta^+_G} \lambda(H_{\alpha})^2.$$
Solving for $c_{\lambda}^2$, taking the positive square root, and observing $H_{\alpha}=\alpha^{\vee}$ is the coroot, the lemma follows.
\end{proof}

In the last lemma, we observed that $c_{\lambda}$ is a constant multiple of the polynomial $\pi_{G/L}^{\vee}=\prod_{\alpha\in \Delta_G^+\setminus \Delta_L^+} \alpha^{\vee}$ in a neighborhood of $\xi$. Let $D(\mathfrak{h}^*)$ denote the Weyl algebra of polynomial coefficient differential operators on $\mathfrak{h}^*$. Given $D\in D(\mathfrak{h}^*)$, we may evaluate $D$ at $\lambda\in \mathfrak{h}^*$ and get a distribution $D(\lambda)$ (ie. $D(\lambda)(f):=(Df)(\lambda)$). We write $D(\lambda)=0$ if $D(\lambda)$ is the zero distribution. 

\begin{lemma} The elements $\partial(\pi_L),\pi_{G/L}^{\vee}\in D(\mathfrak{h}^*)$ commute at the point $\xi$. More precisely,
$$[\partial(\pi_L),\pi_{G/L}^{\vee}](\xi)=0.$$
\end{lemma}
\begin{proof} Suppose $S\subset \Delta^+_L$ is a subset, define $\pi_S=\prod_{\alpha\in S}\alpha$, and let $w\in \text{Aut}\mathfrak{h}$ be a linear automorphism. Then for purely formal reasons, 
$$\langle \partial(w\pi_S),\ w\pi_{G/L}^{\vee}\rangle(w\xi)=\langle \partial(\pi_{S}),\ \pi_{G/L}^{\vee}\rangle(\xi).$$
(If $D\in S(\mathfrak{h}^*)$ is a differential operator on $\mathfrak{h}^*$, $p\in S(\mathfrak{h})$ is a polynomial on $\mathfrak{h}^*$, and $\zeta\in \mathfrak{h}^*$ is a point, then $\langle D,p\rangle(\zeta):=(Dp)(\zeta)$ denotes differentiating the polynomials $p$ by $D$ and evaluating at $\zeta$). Now suppose $w\in W_L$ where $W_L$ is the Weyl group of root system $\Delta_L$. Then $w\xi=\xi$ and $w\pi_{G/L}^{\vee}=\pi_{G/L}^{\vee}$. Hence,
$$\langle \partial(w\pi_S),\ \pi_{G/L}^{\vee}\rangle(\xi)=\langle \partial(\pi_{S}),\ \pi_{G/L}^{\vee}\rangle(\xi)\ \ \ \ \ (*)$$
for all $S\subset \Delta_L^+$ and all $w\in W_L$. Now define
$$wS:=\{\alpha\in \Delta_L^+|\ \alpha=\pm w\beta\ \text{and}\ \beta\in S\}.$$
This defines an action of $W_L$ on the set of subsets of $\Delta_L^+$. Let $W_S$ be the stabilizer of $S$ in $W_L$. Then
$$\partial(\pi_L)\pi_{G/L}^{\vee}=\sum_{\substack{W-\text{orbits\ of}\\ \text{subsets}\ S\subset \Delta_L^+}} \frac{|W_S|}{|W_L|}\sum_{w\in W_L} \epsilon_L(w)\langle \partial(w\pi_S), \pi_{G/L}^{\vee}\rangle \partial(w\pi_{S^{c}}).$$
Here $\epsilon_L$ is the sign representation of $W_L$ and $S^{c}$ is the compliment of $S$ in $\Delta_L^+$. 
Moreover, the notation $\langle \partial(w\pi_S),\pi^{\vee}_{G/L} \rangle$ simply means that we differentiate the polynomial $\pi_{G/L}^{\vee}$ by $\partial(w\pi_S)$. Evaluating at $\xi$ and applying (*), our sum becomes
$$\sum_{\substack{W-\text{orbits\ of}\\ \text{subsets}\ S\subset \Delta_L^+}} \frac{|W_S|}{|W_L|}\langle \partial(\pi_S), \pi_{G/L}^{\vee}\rangle(\xi)\ \ \partial\left(\sum_{w\in W_L} \epsilon_L(w)w\pi_{S^{c}}\right)\Big|_{\xi}.$$
Note that the polynomial $\sum \epsilon_L(w) w\pi_{S^{c}}$ is skew with respect to $W_L$. Thus, $\pi_L$ must divide this polynomial. However, if $S\neq \emptyset$, then the degree of $\sum \epsilon_L(w) w\pi_{S^{c}}$ is less than the degree of $\pi_L$. Thus, our polynomial must be the zero polynomial if $S\neq \emptyset$. If $S=\emptyset$, then $\sum \epsilon_L(w) w\pi_{S^{c}}=|W_L|\pi_L$. Plugging this back into the above expression, we end up with $\pi_{G/L}^{\vee}(\xi) \partial(\pi_L)|_{\xi}$ as desired.
\end{proof}

\indent Now, we prove Theorem 3.1. If $f\in C_c^{\infty}(\mathfrak{g})$, then applying Lemmas 3.2 and 3.3 yields
$$\lim_{\lambda\in C,\ \lambda\rightarrow \xi} \partial(\pi_L)|_{\lambda}\langle \mathcal{O}^G_{\lambda}, f\rangle=\lim_{\lambda\in C,\ \lambda\rightarrow \xi} \partial(\pi_L)|_{\lambda} c_{\lambda}\int_{G/L} \langle \mathcal{O}^L_{\lambda}, g\cdot f\rangle dg$$
$$=\lim_{\lambda\in C,\ \lambda\rightarrow \xi} \partial(\pi_L)|_{\lambda} \int_{G/L} \langle \mathcal{O}^L_{\lambda}, g\cdot f\rangle dg
=\int_{G/L} \lim_{\lambda\in C,\ \lambda\rightarrow \xi} \partial(\pi_L)|_{\lambda}\langle \mathcal{O}^L_{\lambda}, g\cdot f\rangle dg.$$
Applying theorem 2.1, we have
$$\lim_{\lambda\in C,\ \lambda\rightarrow \xi}  \partial(\pi_L)|_{\lambda}\langle \mathcal{O}^L_{\lambda}, g\cdot f\rangle=i^{r(L)}(-1)^{q(L)}|W(L,H)|f(g\cdot \xi).$$
Since we normalized the measure on $G/L\cong \mathcal{O}^G_{\xi}$ to be the canonical one, when we integrate both sides over $G/L$, we get
$$\lim_{\lambda\in C,\ \lambda\rightarrow \xi} \partial(\pi_L)|_{\lambda}\langle \mathcal{O}^G_{\lambda}, f\rangle=i^{r(L)}(-1)^{q(L)}|W(L,H)|\langle \mathcal{O}_{\xi}^G, f\rangle$$
as desired.

\section{Applications of a Lemma of Rao and a Limit Formula of Barbasch}
\indent Identify $\mathfrak{g}\cong \mathfrak{g}^*$ via a $G$-equivariant isomorphism. Let $\mathcal{O}_X$ be a nilpotent orbit in $\mathfrak{g}^*\cong \mathfrak{g}$, and let $\{X,H,Y\}$ be an $\mathfrak{sl}_2$-triple with nilpositive element $X$. Put $S_X=X+Z_{\mathfrak{g}}(Y)$. 

\begin{proposition} Let $\nu\in \mathfrak{g}^*$ be semisimple. Then 
$$\lim_{t\rightarrow 0^+} \mathcal{O}_{t\nu}=\sum_{\substack{\mathcal{O}_X\ \text{nilpotent}\\ \mathcal{O}_{\nu}\cap S_X\ \text{finite}}}\#(\mathcal{O}_{\nu}\cap S_X)\mathcal{O}_X.$$
\end{proposition}

\indent The fact that the limit of distributions on the left hand side exists was known to Harish-Chandra. It follows from Lemma 22 of \cite{HC3} together with Theorem 3.1 above. Further, the limit on the left hand side is easily seen to be a non-negative, homogeneous distribution supported on the nilcone. It is well-known that every non-negative distribution is a Radon measure. Thus, it follows that the limit must be a linear combination of canonical invariant measures on nilpotent coadjoint orbits.\\
To verify the precise nature of the sum, we recall an unpublished lemma of Rao. Nothing we say is particularly deep; hence, we leave the verifications of these well-known facts to the reader.\\
\indent Observe that the map 
$$\phi: G\times S_X\rightarrow \mathfrak{g}^*$$ 
given by $\phi:(g,\xi)\mapsto g\cdot \xi$ is a submersion. In particular, every orbit $\mathcal{O}_{\nu}\subset G\cdot S_X$ is transverse to $S_X$, and $G\cdot S_X\subset \mathfrak{g}^*$ is open.\\
\indent Fix a Haar measure on $G$. This choice determines a Lebesgue measure on $\mathfrak{g}\cong \mathfrak{g}^*$. The direct sum decomposition $\mathfrak{g}=[\mathfrak{g},X]\oplus Z_{\mathfrak{g}}(Y)$ and the canonical measure on $\mathcal{O}_X$ determine a Lebesgue measure on $S_X$. Further, given $\nu\in \mathfrak{g}^*$, denote by $\mathcal{F}_{\nu}$ the fiber over $\nu$ under the map $\phi$. If $g\cdot \xi=\nu$, then we have an exact sequence
$$0\rightarrow T_{\nu}^*(G\cdot S_X) \rightarrow T_{(g,\xi)}^*(G\times S_X)\rightarrow T_{(g,\xi)}^*\mathcal{F}_{\nu}\rightarrow 0.$$
This exact sequence together with the above remarks and our choice of Haar measure on $G$ determine a smooth measure on $\mathcal{F}_{\nu}$. Moreover, integration against these measures on the fibers of $\phi$ yields a continuous surjective map
$$\phi_*:\ C_c^{\infty}(G\times S_X)\longrightarrow C_c^{\infty}(G\cdot S_X).$$
\noindent Dualizing, we get an injective pullback map on distributions
$$\phi^*:\ D(G\cdot S_X)\rightarrow D(G\times S_X).$$
\noindent Now, we are ready to state Rao's lemma.

\begin{lemma} [Rao] If $\nu\in S_X$, then there exists a smooth measure $m_{\nu,X}$ on $\mathcal{O}_{\nu}\cap S_X$ such that 
$$\phi^*(\mathcal{O}_{\nu})=m_G\otimes m_{\nu,X}.$$
Here $m_G$ denotes the fixed choice of Haar measure on $G$. Although $\phi^*$ depends on this choice of Haar measure, $m_{\nu,X}$ does not.
\end{lemma}

\indent One can write down $m_{\nu,X}$ by giving a top dimensional form on $\mathcal{O}_{\nu}\cap S_X$, well-defined up to sign. Essentially, we just divide the canonical measure on $\mathcal{O}_{\nu}$ by the canonical measure on $\mathcal{O}_X$. More precisely, the composition of the inclusion $[\mathfrak{g},\nu]\hookrightarrow \mathfrak{g}$ and the projection defined by the decomposition $\mathfrak{g}=[\mathfrak{g},X]\oplus Z_{\mathfrak{g}}(Y)$ yields a map
$$T_{\nu}\mathcal{O}_{\nu}\cong [\mathfrak{g},\nu]\rightarrow [\mathfrak{g},X]\cong T_X\mathcal{O}_X.$$
This surjection pulls back to an exact sequence
$$0\rightarrow T_X^*\mathcal{O}_X\rightarrow T_{\nu}^*\mathcal{O}_{\nu}\rightarrow T^*_{\nu}(\mathcal{O}_{\nu}\cap S_X)\rightarrow 0.$$
\noindent The canonical measures on $\mathcal{O}_{\nu}$ and $\mathcal{O}_X$ determine top dimensional alternating tensors up to sign on $T_{\nu}^*\mathcal{O}_{\nu}$ and $T_X^*\mathcal{O}_X$. Hence, our exact sequence gives a top dimensional, alternating tensor on $T^*_{\nu}(\mathcal{O}_{\nu}\cap S_X)$, well-defined up to sign.\\

\indent Next, we need a proposition of Barbasch and Vogan. If $\nu\in \mathfrak{g}^*$, then define
$$\mathcal{N}_{\nu}=\overline{\cup_{t>0} \mathcal{O}_{t\nu}}\cap \mathcal{N}.$$

\begin{proposition} [Barbasch and Vogan] Suppose $\nu\in \mathfrak{g}^*$. If $\mathcal{O}_X$ is a nilpotent orbit, then $\mathcal{O}_X\subset \mathcal{N}_{\nu}$ if, and only if $\mathcal{O}_{\nu}\cap S_X\neq \emptyset$. Further, $\mathcal{O}_X\cap S_X=\{X\}$ for any nilpotent orbit $\mathcal{O}_X$.
\end{proposition}

A proof of the last sentence can be found on the top half of page 48 of \cite{BV}. We recall the proof of the first part of the proposition because it gives us an excuse to introduce some notation we will use later. Recall $G\cdot S_X\subset \mathfrak{g}^*$ is an open subset containing $\mathcal{O}_X$; thus, $\mathcal{O}_X\subset \mathcal{N}_{\nu}$ iff $\mathcal{O}_{t\nu}\cap S_X\neq \emptyset$ for sufficiently small $t>0$. However, if $\gamma_t=\text{exp}(-\frac{1}{2}(\log(t))H)$, then  $$\mathcal{O}_{t\nu}\cap S_X=t\gamma_t(\mathcal{O}_{\nu}\cap S_X).$$ In particular, $\mathcal{O}_{\nu}\cap S_X \neq \emptyset$ iff $\mathcal{O}_{t\nu}\cap S_X\neq \emptyset$ for any $t>0$.\\
\indent Now, back to the proof of Proposition 4.1. We know that the limit $\lim_{t\rightarrow 0^+} \mathcal{O}_{t\nu}$ exists and is a linear combination of nilpotent coadjoint orbits. Observe that the support of the distribution must be contained in $\mathcal{N}_{\nu}$. Hence, by Proposition 4.3, we are summing over orbits $\mathcal{O}_X$ such that $\mathcal{O}_{\nu}\cap S_X\neq \emptyset$. Checking the homogeneity degree, we observe that these orbits $\mathcal{O}_X$ must satisfy $\dim \mathcal{O}_X=\dim \mathcal{O}_{\nu}$. Finally, since every orbit intersecting $S_X$ is transverse to $S_X$, the condition $\dim \mathcal{O}_X=\dim \mathcal{O}_{\nu}$ is equivalent to $\mathcal{O}_{\nu}\cap S_X$ being finite. Thus, we realize that only the orbits satisfying $\#(\mathcal{O}_{\nu}\cap S_X)<\infty$ can occur in the limit $\lim_{t\rightarrow 0^+}\mathcal{O}_{t\nu}$.\\
\indent Now, to compute the coefficients of these orbits, we use Rao's lemma. To finish the proof, it is enough to show
$$\lim_{t\rightarrow 0^+} \mathcal{O}_{t\nu}=\#(\mathcal{O}_{\nu}\cap S_X)\mathcal{O}_X\ \text{on}\ G\cdot S_X.$$
To do this, we apply the injective map $\phi^*$ and Rao's lemma to reduce the question to proving
$$\lim_{t\rightarrow 0^+} m_{t\nu,X}=\#(\mathcal{O}_{\nu}\cap S_X)\delta_X.$$
Observe that the measure $m_{t\nu,X}$ is supported on the finite set $\mathcal{O}_{t\nu}\cap S_X$. If $\mathcal{O}_{\nu}\cap S_X=\{\nu_1,\ldots,\nu_k\}$, then $\mathcal{O}_{t\nu}\cap S_X=\{t\gamma_t\nu_1,\ldots,t\gamma_t\nu_k\}$. We observe $t\gamma_t\nu_j\rightarrow X$ as $t\rightarrow 0^+$ and $m_{t\nu,X}|_{t\gamma_t\nu_j}\rightarrow \delta_X$ as $t\rightarrow 0^+$. These facts together imply the desired relation $\lim_{t\rightarrow 0} m_{t\nu,X}=\#(\mathcal{O}_{\nu}\cap S_X)\delta_X$. This completes the proof of Proposition 4.1.\\
\\
Next, we record a couple of useful corollaries to Proposition 4.1. Suppose $\nu\in \mathfrak{g}^*$ is semisimple, and write
$$\lim_{t\rightarrow 0^+} \mathcal{O}_{t\nu}=\sum n_G(\mathcal{O},\nu) \mathcal{O}.$$
If $\mathcal{O}$ is an orbit occuring in the sum, then we let $\mathcal{O}_{\mathbb{C}}=\text{Int}\mathfrak{g}_{\mathbb{C}}\cdot \mathcal{O}$ denote its complexification, and we denote by $n_{\operatorname{Int}\mathfrak{g}_{\mathbb{C}}}(\mathcal{O}_{\mathbb{C}},\nu)$ the coefficient in the corresponding limit formula of $\text{Int}\mathfrak{g}_{\mathbb{C}}$-orbits. 

\begin{corollary}  The coefficients $n_G(\mathcal{O},\nu)$ are non-negative integers. Moreover, $n_G(\mathcal{O},\nu)\leq n_{\operatorname{Int}\mathfrak{g}_{\mathbb{C}}}(\mathcal{O}_{\mathbb{C}},\nu)$. 
\end{corollary}

The coefficient $n_G(\mathcal{O},\nu)$ is a non-negative integer because it is the cardinality of a finite set by Proposition 4.1. Note $n_G(\mathcal{O},\nu)$ is the cardinality of the finite set $\mathcal{O}_{\nu}\cap (X+Z_{\mathfrak{g}}(Y))$ while $n_{\operatorname{Int}\mathfrak{g}_{\mathbb{C}}}(\mathcal{O}_{\mathbb{C}},\nu)$ is the cardinality of the set $(\text{Int}\mathfrak{g}_{\mathbb{C}}\cdot \mathcal{O}_{\nu})\cap (X+Z_{\mathfrak{g}_{\mathbb{C}}}(Y))$. Since the former set is contained in the later set, we deduce $n_G(\mathcal{O},\nu)\leq n_{\operatorname{Int}\mathfrak{g}_{\mathbb{C}}}(\mathcal{O}_{\mathbb{C}},\nu)$.

\begin{corollary} Let $\nu\in \mathfrak{g}^*$, let $L=Z_G(\nu)$, and suppose $\mathcal{O}$ is a nilpotent orbit with $n_G(\mathcal{O},\nu)\neq 0$. After conjugating by $G$, we may assume $\nu\in S_X$. There exists a maximal compact subgroup $K\subset G$ such that $Z_K\{X,H,Y\}\subset Z_G\{X,H,Y\}$ and $K\cap L\subset L$ are maximal compact subgroups. If $K$ is such a group, then
$$\left|Z_K(X)/Z_{K\cap L}(X)\right|\leq n_G(\mathcal{O},\nu).$$
\end{corollary}

First, if $n_G(\mathcal{O},\nu)\neq 0$, then $Z_G\{X,H,Y\}/Z_L\{X,H,Y\}$ acts faithfully on the finite set $\mathcal{O}_{\nu}\cap S_X$ by Corollary 3.1.5. In particular, we have a chain of reductive groups
$$G\supset L\supset Z_G\{X,H,Y\}_0$$
where $Z_G\{X,H,Y\}$ is the identity component of $Z_G\{X,H,Y\}$. Recall that any compact subroup of a reductive Lie group is contained in a maximal compact subgroup of a reductive Lie group. It follows from this fact that there exists a maximal compact subgroup $K\subset G$ such that
$$K\cap L\subset L,\ K\cap Z_G\{X,H,Y\}_0\subset Z_G\{X,H,Y\}_0$$
are maximally compact subgroups. But, it is not difficult to see that whenever $K\subset G$ is a maximally compact subgroup, we have $K\cap Z_G\{X,H,Y\}\subset Z_G\{X,H,Y\}$ is maximally compact. This proves the first statement of the proposition.\\
\\
Note that $Z_K\{X,H,Y\}$ acts on the finite set $\mathcal{O}_{\nu}\cap S_X$ with stabilizer $Z_{K\cap L}\{X,H,Y\}$. Thus, we deduce $\left|Z_K\{X,H,Y\}/Z_{K\cap L}\{X,H,Y\}\right|\leq n_G(\mathcal{O},\nu)$. Hence, to prove the corollary, it is enough to show that the injection 
$$Z_K\{X,H,Y\}/Z_{K\cap L}\{X,Y,H\} \hookrightarrow Z_K(X)/Z_{K\cap L}(X)$$ is in fact a surjection.\\
To do this, we use two commutative diagrams. First, we have

$$\begin{CD}
Z_K\{X,H,Y\}/Z_{K\cap L}\{X,H,Y\}       @>>>       Z_G\{X,H,Y\}/Z_L\{X,H,Y\}\\
@VVV                                                @VVV\\
Z_K(X)/Z_{K\cap L}(X)                  @>>>        Z_G(X)/<Z_G(X)^0,Z_L(X)>
\end{CD}$$
where $Z_G(X)^0$ denotes the identity component of $Z_G(X)$ and $<Z_G(X)^0,Z_L(X)>$ denotes the group generated by $Z_G(X)^0$ and $Z_L(X)$. The top arrow is a surjection because the maximal compact subgroup $Z_K\{X,H,Y\}$ meets every component of the reductive Lie group $Z_G\{X,H,Y\}$. The arrow on the right is a surjection because every component of $Z_G(X)$ meets the Levi factor $Z_G\{X,H,Y\}$. Hence, to show that the arrow on the left is a surjection, it is enough to show that the botton arrow is an injection.\\
\indent To verify this last statement, we need some notation and a second commutative diagram. Find a real, reductive algebraic group $G_{\mathbb{R}}$ and a map $p:G\rightarrow G_{\mathbb{R}}$ with open image and finite kernel. Choose a maximal compact subgroup $K_{\mathbb{R}}\subset G_{\mathbb{R}}$ such that $p(K)\subset K_{\mathbb{R}}$, and choose a Levi subgroup $L_{\mathbb{R}}\subset G_{\mathbb{R}}$ such that $p:L\rightarrow L_{\mathbb{R}}$ has open image and finite kernel. Let $L_{\mathbb{C}}$ be the complexification of $L_{\mathbb{R}}$, and let $U\subset G_{\mathbb{C}}$ be a maximal compact subgroup with $K_{\mathbb{R}}=U\cap G_{\mathbb{R}}$. Choose a parabolic subgroup $P_{\mathbb{C}}\subset G_{\mathbb{C}}$ with Levi factor $L_{\mathbb{C}}$. Then we have the following commutative diagram.

$$\begin{CD}
Z_K(X)/Z_{K\cap L}(X)    @>>>     Z_G(X)/<Z_G(X)^0,Z_L(X)>\\
@VVV															@VVV\\
Z_U(X)/Z_{U\cap L_{\mathbb{C}}}(X)   @>>>    Z_{G_{\mathbb{C}}}(X)/Z_{P_{\mathbb{C}}}(X)
\end{CD}$$

The left and bottom maps are easily seen to be injective; hence the top map also must be injective. The corollary follows.\\
\\
Next, we recall a proposition of Dan Barbasch \cite{Ba}, which provides an explicit formula for $n_{\operatorname{Int}\mathfrak{g}_{\mathbb{C}}}(\mathcal{O}_{\mathbb{C}},\nu)$. Let $\nu\in \mathfrak{g}_{\mathbb{C}}^*$ be a semisimple element, let $L=Z_{\text{Int}\mathfrak{g}_{\mathbb{C}}}(\nu)$, and let $\mathfrak{l}\subset \mathfrak{p}$ be a parabolic containing $\mathfrak{l}=\text{Lie}(L)$. Suppose $X\in (\mathfrak{g}/\mathfrak{p})^*$ is a nilpotent element such that $\mathcal{O}^{\text{Int}\mathfrak{g}_{\mathbb{C}}}_X\cap (\mathfrak{g}/\mathfrak{p})^*\subset (\mathfrak{g}/\mathfrak{p})^*$ is open.

\begin{proposition} [Barbasch] We have the limit formula
$$\lim_{t\rightarrow 0^+} \mathcal{O}^{\text{Int}\mathfrak{g}_{\mathbb{C}}}_{t\nu}=\left|Z_{\text{Int}\mathfrak{g}_{\mathbb{C}}}(X)/Z_{P_{\mathbb{C}}}(X)\right|\mathcal{O}^{\text{Int}\mathfrak{g}_{\mathbb{C}}}_X$$
where $P_{\mathbb{C}}=N_{\text{Int}\mathfrak{g}_{\mathbb{C}}}(\mathfrak{p}_{\mathbb{C}})$.
\end{proposition}

In particular, if $Z_{\text{Int}\mathfrak{g}_{\mathbb{C}}}(X)$ is connected, then $\lim_{t\rightarrow 0^+} \mathcal{O}^{\text{Int}\mathfrak{g}_{\mathbb{C}}}_{t\nu}=\mathcal{O}^{\text{Int}\mathfrak{g}_{\mathbb{C}}}_X$. By a computation of Springer-Steinberg explained on page 88 of \cite{CM}, this is true when $\mathfrak{g}_{\mathbb{C}}\cong \mathfrak{gl}(n,\mathbb{C})$. Moreover, every nilpotent coadjoint orbit for $\text{GL}(n,\mathbb{C})$ can be written as such a limit by a result of Ozeki and Wakimoto explained in section 7.2 of \cite{CM}. Further, it also follows from results in 7.2 and Barbasch's limit formula that two limit formulas 
$$\lim_{t\rightarrow 0^+}\mathcal{O}^{\text{GL}(n,\mathbb{C})}_{t\xi_1},\ \lim_{t\rightarrow 0^+}\mathcal{O}^{\text{GL}(n,\mathbb{C})}_{t\xi_2}$$
yield the same nilpotent orbit if and only if $Z_{\text{GL}(n,\mathbb{C})}(\xi_1)$ and $Z_{\text{GL}(n,\mathbb{C})}(\xi_2)$ are conjugate. In the next corollary, we observe that these results also hold for $\text{GL}(n,\mathbb{R})$.\\
To state it, we define the moment map. Let $G$ be a reductive Lie group and let $\mathcal{P}$ be a conjugacy class of parabolic subgroups of $G$. Then 
$$T^*\mathcal{P}=\{(\mathfrak{p},\xi)|\ \mathfrak{p}\in \mathcal{P},\ \xi\in (\mathfrak{g}/\mathfrak{p})^*\subset \mathfrak{g}^*\}$$
and the moment map is defined by $\mu(\mathfrak{p},\xi)=\xi$. (Of course, the moment map can be defined for any Hamiltonian action of a Lie group; however, we do not need the more general definition here).

\begin{proposition} There exists a bijection between conjugacy classes of Levi factors of parabolic subgroups of $\text{GL}(n,\mathbb{R})$ and nilpotent coadjoint orbits for $\text{GL}(n,\mathbb{R})$. Suppose $\mathcal{L}$ is a conjugacy class of Levi factors, and let $\mathcal{P}$ be the conjugacy class of parabolics containing $\mathcal{L}$. Then the orbit $\mathcal{O}_{\mathcal{L}}$ is the unique open, dense orbit in the image of the moment map of the real generalized flag variety
$$\mathcal{O}_{\mathcal{L}}\subset \mu(T^*\mathcal{P}).$$
Alternately, we may choose $\xi\in \mathfrak{gl}(n,\mathbb{R})^*$ such that $Z_{\text{GL}(n,\mathbb{R})}(\xi)=L$. Then $\mathcal{O}_{\mathcal{L}}$ is also characterized by the limit formula
$$\lim_{t\rightarrow 0^+} \mathcal{O}_{t\xi}=\mathcal{O}_{\mathcal{L}}.$$ 
\end{proposition}

The first $\text{GL}(n,\mathbb{R})$ statement follows immediately from the corresponding $\text{GL}(n,\mathbb{C})$ statement together with the fact that every nilpotent coadjoint $\text{GL}(n,\mathbb{C})$-orbit has an unique real form and the fact that $\mathcal{O}\cap (\mathfrak{gl}(n,\mathbb{R}))/\mathfrak{p})^*$ is dense if and only if $(\text{GL}(n,\mathbb{C})\cdot \mathcal{O})\cap (\mathfrak{gl}(n,\mathbb{C}))/\mathfrak{p}_{\mathbb{C}})^*$ is dense. It follows from Corollary 4.4 and the above $\text{GL}(n,\mathbb{C})$ remarks that $\lim_{t\rightarrow 0^+} \mathcal{O}_{t\xi}$ is either zero or $\mathcal{O}_{\mathcal{L}}$. In the last two sections of this article, we will use the results of the first two sections to explicitly compute $\lim_{t\rightarrow 0^+} \widehat{\mathcal{O}_{t\xi}}$. We will observe that the answer is non-zero. This will complete the proof of the proposition and compute the Fourier transform of the nilpotent orbit $\mathcal{O}_{\mathcal{L}}$.

\section{Limit Formulas for Even Nilpotent Orbits}

In \cite{Bo1}, Bozicevic proves the following limit formula for an even nilpotent orbit. 

\begin{proposition} [Rao,\ Bozicevic]
Suppose $\mathcal{O}_X$ is an even nilpotent orbit, let $\{X,H,Y\}$ be an $\mathfrak{sl}_2$-triple containing $X$, and let $Z=X-Y$. Then
$$\lim_{t\rightarrow 0^+} \mathcal{O}_{tZ}=\mathcal{O}_X.$$
\end{proposition}

This formula was first proved by Rao in an unpublished paper. Bozicevic's formula has a coefficient in front of the $\mathcal{O}_X$. In fact, this coefficient is one. Bozicevic's proof involves deep results of Schmid and Vilonen. In this section, we show how this formula follows easily from the far more elementary results of the last section.\\
\indent First, let $\mathfrak{p}_{\mathbb{C}}$ be the sum of non-negative eigenspaces for $\text{ad}_H$ on $\mathfrak{g}_{\mathbb{C}}$. Then $\mathcal{O}^{\text{Int}\mathfrak{g}_{\mathbb{C}}}_X\cap (\mathfrak{g}_{\mathbb{C}}/\mathfrak{p}_{\mathbb{C}})^*\subset (\mathfrak{g}_{\mathbb{C}}/\mathfrak{p}_{\mathbb{C}})^*$ is open and we may apply Barbasch's result, proposition 4.6. Further, a result of Barbasch-Vogan and Kostant explained on page 50 of \cite{CM} implies $$Z_{\text{Int}\mathfrak{g}_{\mathbb{C}}}(X)/Z_{P_{\mathbb{C}}}(X)\cong Z_{\text{Int}\mathfrak{g}_{\mathbb{C}}}\{X,H,Y\}/Z_{P_{\mathbb{C}}}\{X,H,Y\}.$$
But, $Z_{\text{Int}\mathfrak{g}_{\mathbb{C}}}\{X,H,Y\}\subset Z_{\text{Int}\mathfrak{g}_{\mathbb{C}}}(Z)\subset P_{\mathbb{C}}$. Hence, our coefficient is one and we have
$$\lim_{t\rightarrow 0^+} \mathcal{O}^{\text{Int}\mathfrak{g}_{\mathbb{C}}}_{tZ}=\mathcal{O}_X^{\text{Int}\mathfrak{g}_{\mathbb{C}}}.$$
\indent Now, we need to prove a real version of this limit formula. By Corollary 4.4, we know that we must have 
$\lim_{t\rightarrow 0^+} \mathcal{O}_{tZ}=\sum \mathcal{O}_{X'}$
where the sum is over some subset of real forms of $\mathcal{O}^{\text{Int}\mathfrak{g}_{C}}_X$. We know $\mathcal{O}_X$ must occur by proposition 4.1 and the observation $Z=X-Y\in X+Z_{\mathfrak{g}}(Y)$. Now suppose $\mathcal{O}_{X'}$ is some other real form of $\mathcal{O}_X^{\text{Int}\mathfrak{g}_{\mathbb{C}}}$ occuring in our limit formula. Given an $\mathfrak{sl}_2$-triple $\{X',H',Y'\}$ containing $X'$, we must have $\mathcal{O}_Z\cap (X'+Z_{\mathfrak{g}}(Y'))\neq \emptyset$. But, $Z'=X'-Y'\in X'+Z_{\mathfrak{g}}(Y')$ and $\mathcal{O}_{Z'}^{\text{Int}\mathfrak{g}_{\mathbb{C}}}\cap (X'+Z_{\mathfrak{g}_{\mathbb{C}}}(Y'))$ has one element by proposition 4.1 and the above $\text{Int}\mathfrak{g}_{\mathbb{C}}$-limit formula. Further, it was proven by Rao (unpublished) that $Z'=X'-Y'=X-Y=Z$ only if $X$ and $X'$ are conjugate (details of his elementary argument can be found on page 146 of \cite{CM}). Thus, we cannot have $Z\in X'+Z_{\mathfrak{g}}(Y')$ and no other real forms can occur in our limit formula. The proposition follows.

\section{Fourier Transforms of Semisimple Coadjoint Orbits for $\text{GL}(n,\mathbb{R})$}

Let $G=\text{GL}(n,\mathbb{R})=\text{GL}(2m+\delta,\mathbb{R})$ where $\delta=0\ \text{or}\ 1$, and let $\mathfrak{g}=\text{Lie}(G)$. Fix a fundamental Cartan $\mathfrak{h}_0\subset \mathfrak{g}$, and enumerate its imaginary roots $$\{\alpha_1,\ldots,\alpha_m,-\alpha_1,\ldots,-\alpha_m\}.$$ 
Let $\mathfrak{h}_k$ be the Cartan obtained by applying Cayley transforms through the roots $\alpha_1,\ldots,\alpha_k$. Then $\mathfrak{h}_0,\ldots,\mathfrak{h}_m$ is a set of representatives of the conjugacy classes of Cartan subgalebras of $\mathfrak{g}$. In what follows, we will use these fixed Cayley transforms to identify $(\mathfrak{h}_k)_{\mathbb{C}}\cong (\mathfrak{h}_l)_{\mathbb{C}}$ (and all roots, coroots of $\mathfrak{h}_k$ with roots, coroots of $\mathfrak{h}_l$) without further comment.\\
\indent Let $\Delta(\mathfrak{h}_l)$ (resp. $\Delta_{\text{imag.}}(\mathfrak{h}_l)$, $\Delta_{\text{real}}(\mathfrak{h}_l)$, $\Delta_{\text{cx.}}(\mathfrak{h}_l)$) denote the set of all (resp. imaginary, real, complex) roots of $\mathfrak{g}$ with respect to $\mathfrak{h}_l$. Choose a component $C_m\subset \mathfrak{h}_m'$, and define $\Delta^+$ to be the set of roots $\alpha$ such that $\alpha(X)>0$ for all $X\in C_m$. This fixes a choice of positive roots for $\mathfrak{g}$ with respect to $\mathfrak{h}_l$ for every $l$. Denote by $\Delta^+(\mathfrak{h}_l)$ (resp. $\Delta^+_{\text{imag.}}(\mathfrak{h}_l)$, $\Delta^+_{\text{real}}(\mathfrak{h}_l)$, $\Delta^+_{\text{cx.}}(\mathfrak{h}_l$)) the set of all (resp. imaginary, real, complex) positive roots of $\mathfrak{g}$ with respect to $\mathfrak{h}_l$. Now, choose a regular element $\lambda\in \mathfrak{h}_k^*$ satisfying:

\noindent (a) If $\alpha\in \Delta^+_{\text{imag.}}(\mathfrak{h}_k)$ is a positive, imaginary root of $\mathfrak{h}_k$, then 
$$\langle \lambda, i\alpha^{\vee}\rangle<0.$$
\noindent (b) If $\beta\in \Delta^+_{\text{real}}(\mathfrak{h}_k)$ is a real, positive root of $\mathfrak{h}_k$, then 
$$\langle \lambda, \beta^{\vee}\rangle<0.$$
\noindent Moreover, define 
$$C_l(e)=\{X\in \mathfrak{h}_l'|\alpha(X)>0\ \forall \alpha\in \Delta^+_{\text{real}}(\mathfrak{h}_k)\},$$
and for every $u\in W_{\text{real}}(\mathfrak{h}_l)$, define
$$C_l(u)=u\cdot C_l(e).$$
Here $W_{\text{real}}(\mathfrak{h}_l)$ denotes the Weyl group of the real roots of $\mathfrak{g}$ with respect to $\mathfrak{h}_l$. Note $C_m(e)=C_m$.\\
\\
Let $W_{\mathbb{C}}$ denote the complex Weyl group of $\mathfrak{g}_{\mathbb{C}}$ with respect to $(\mathfrak{h}_l)_{\mathbb{C}}$. Define a subset $W_{k,l}\subset W_{\mathbb{C}}$ to be the set of $w\in W_{\mathbb{C}}$ satisfying:\\
\noindent (i) If $\alpha\in \Delta_{\text{imag.}}(\mathfrak{h}_l)$, then $w^{-1}\alpha\in \Delta_{\text{imag.}}(\mathfrak{h}_k)$.\\
\noindent (ii) If $\alpha\in \Delta_{\text{cx.}}(\mathfrak{h}_l)$, then $w^{-1}\alpha\in \Delta_{\text{cx.}}(\mathfrak{h}_k)$.\\
\noindent (iii) If $\alpha\in \Delta^+_{\text{imag.}}(\mathfrak{h}_k)$ and $w\alpha\notin \Delta_{\text{imag.}}(\mathfrak{h}_l)$, then $w\alpha\in \Delta^+_{\text{real}}(\mathfrak{h}_l)$.\\ 
\\
\noindent For each $w\in W_{k,l}$, let $N_{k,l}(w)$ be the number of $\alpha\in \Delta^+_{\text{imag.}}(\mathfrak{h}_l)$ such that $w^{-1}\alpha\notin \Delta^+_{\text{imag.}}(\mathfrak{h}_k)$. Define
$$\epsilon_{k,l}(w)=(-1)^{N_{k,l}(w)}$$ for every $w\in W_{k,l}$.

\begin{proposition} If $l\geq k$, then 
$$\widehat{\mathcal{O}_{\lambda}}|_{C_l(e)}=\frac{2^{l-k}\sum_{w\in W_{k,l}} \epsilon_{k,l}(w) e^{iw\lambda}}{\pi}.$$
\noindent If $l<k$, then $\widehat{\mathcal{O}_{\lambda}}$ vanishes on $\mathfrak{h}_l'$. Here $\pi=\prod_{\alpha\in \Delta^+(\mathfrak{h}_l)}\alpha$ as usual.
\end{proposition} 

\noindent In the above discussion, we chose positive roots and components, and then we chose the regular, semisimple element $\lambda$ in a compatible way. One can also choose $\lambda\in \mathfrak{h}_k^*$ regular, semisimple and then choose roots and components compatibly so that the proposition holds.
\begin{proof} Let $\theta$ be a Cartan involution fixing $\mathfrak{h}_k$, and decompose $\mathfrak{h}_k=\mathfrak{t}\oplus \mathfrak{a}$ where $\mathfrak{t}$ is the $+1$ eigenspace of $\theta$ and $\mathfrak{a}$ is the $-1$ eigenspace of $\theta$. Set $M=Z_G(\mathfrak{a})$, and note 
$$M\cong \text{GL}(2,\mathbb{R})^{m-k}\times (\mathbb{R}^{\times})^{2k+\delta}.$$
The identity component of $M$ is $M_0\cong \text{GL}^+(2,\mathbb{R})^{m-k}\times (\mathbb{R}^{\times}_+)^{2k+\delta}.$
\noindent We will compute $\widehat{\mathcal{O}_{\lambda}^G}$ by first computing $\widehat{\mathcal{O}_{\lambda}^M}$ and then applying Harish-Chandra descent.\\

Let $\Delta_M(\mathfrak{h}_l)$ be the set of roots of $M$ with respect to $\mathfrak{h}_l$, and let $\pi_M$ be the product of the roots of $M$ with respect to $\mathfrak{h}_l$ that are positive for $G$.
\noindent Let $W_{\mathbb{R}}^M(\mathfrak{h}_k)$ be the real Weyl group of $M$ with respect to $\mathfrak{h}_k$, and fix $w\in W_{\mathbb{R}}^M(\mathfrak{h}_k)$. Note that $w\pi_M$ is the product of the roots $\alpha$ of $M$ satisfying
$$\langle iw\lambda,\alpha^{\vee}\rangle<0.$$ Then by Theorem 2.4, we have
$$\widehat{\mathcal{O}_{w\lambda}^{M_0}}|_{\mathfrak{h}_k'}=\frac{e^{iw\lambda}}{w\pi_M}.$$
Observe $\mathfrak{h}_l\subset \mathfrak{m}$, and put
$$C_l(e)_M=\{X\in \mathfrak{h}_l'|\ \alpha(X)>0\ \forall\ \text{real\ roots}\ \alpha\in \Delta_M(\mathfrak{h}_l)\cap \Delta_G^+(\mathfrak{h}_l)\}.$$ 
\noindent If $u\in W_{\text{real}}^M(\mathfrak{h}_l)$, define
$$C_l(u)_M=u\cdot C_l(e)_M.$$
Now, decompose $w=w_rw_i$ into its components in the Weyl group of the real roots of $\mathfrak{h}_l$ and the Weyl group of the imaginary roots of $\mathfrak{h}_l$. Checking Harish-Chandra's matching conditions (Theorem 2.7) and using that $\widehat{\mathcal{O}_{w\lambda}^{M_0}}$ is tempered, we observe
$$\widehat{\mathcal{O}_{w\lambda}^{M_0}}|_{C_l(w_r)_M}=\frac{e^{iw\lambda}}{w\pi_M}.$$
Since $w_r$ is in the real Weyl group of $\mathfrak{h}_l$ with respect to $M_0$ and the generalized function we are computing is $M_0$-invariant, we get
$$\widehat{\mathcal{O}_{w\lambda}^{M_0}}|_{C_l(e)_M}=\frac{\epsilon(w_i)e^{iw_i\lambda}}{\pi_M}.$$
\noindent Note that $\mathcal{O}_{\lambda}^M$ is the finite union of the orbits $\mathcal{O}_{w\lambda}^{M_0}$ where $w$ ranges over the real Weyl group $W_{\mathbb{R}}^M(\mathfrak{h}_k)$. Hence,
$$\widehat{\mathcal{O}_{\lambda}^M}|_{C_l(e)_M}=\frac{2^{l-k}\sum_{w\in W_{\text{imag.}}^M(\mathfrak{h}_l)}\epsilon(w)e^{iw\lambda}}{\pi_M}$$
where $W_{\text{imag.}}^M(\mathfrak{h}_l)$ is the Weyl group of the imaginary roots of $\mathfrak{m}$ with respect to $\mathfrak{h}_l$.

Now, we can use Harish-Chandra descent (Lemma 2.5) to compute $\widehat{\mathcal{O}_{\lambda}^G}$. Given $X\in C_l(e)_M$, we must enumerate the $M$-orbits in $\mathcal{O}_X\cap \mathfrak{m}$. First, we choose representatives of the $M$-conjugacy classes of Cartans in $\mathfrak{m}$. For each multi-index $J=(j_1,\ldots, j_{m-l})$ with $1\leq j_1<j_2<\cdots <j_{m-l}\leq m-k$, define $J^c=(r_1,\ldots,r_{l-k})$ to be the complementary indices among $1,\ldots,m-k$. Define $\mathfrak{h}_l^J$ to be the Cartan obtained from $\mathfrak{h}_k\subset \mathfrak{m}$ by applying the Cartan involutions associated to the roots $\alpha_{m-r_s}$ for $s=1,\ldots,l-k$. One sees that the collection $\{\mathfrak{h}_l^J\}$ is a set of representatives for all $M$ conjugacy classes of Cartans in $\mathfrak{m}$ of imaginary rank $m-l$.\\

Now, every $M$ orbit in $\mathcal{O}_X\cap \mathfrak{m}$ must meet at least one of the Cartans $\mathfrak{h}_l^J$. Thus, we have
$$(\mathcal{O}_X\cap \mathfrak{m})/M=\bigcup_{J} (\mathcal{O}_X\cap \mathfrak{h}_l^J)/W_{\mathbb{R}}^M(\mathfrak{h}_l^J)$$ where $W_{\mathbb{R}}^M(\mathfrak{h}_l^J)$ is the real Weyl group of $M$ with respect to $\mathfrak{h}_l^J$.
Moreover, for each Cartan $\mathfrak{h}_l^J$, the $W_{\mathbb{R}}^M(\mathfrak{h}_l^J)$ orbits on $\mathcal{O}_X\cap \mathfrak{h}_l^J$ are in bijection with the cosets $W_{\mathbb{R}}^M(\mathfrak{h}_l^J)\backslash W_{\mathbb{R}}^G(\mathfrak{h}_l^J)$. Thus, we have the formula
$$\widehat{\mathcal{O}^G_{\lambda}}(X)=\sum_J \sum_{u\in W_{\mathbb{R}}^M(\mathfrak{h}_l^J)\backslash W_{\mathbb{R}}^G(\mathfrak{h}_l^J)}\widehat{\mathcal{O}^M_{\lambda}}(uw_JX)|\pi_{G/M}(uw_JX)|^{-1}.$$
\noindent Here $w_J$ is an element of $G$ taking $\mathfrak{h}_l$ to $\mathfrak{h}_l^J$.\\
\indent Note that we get isomorphisms $(\mathfrak{h}_k)_{\mathbb{C}}\cong (\mathfrak{h}_l)_{\mathbb{C}}$ and $(\mathfrak{h}_k)_{\mathbb{C}}\cong (\mathfrak{h}_l^J)_{\mathbb{C}}$ by applying successive Cayley transforms to $\mathfrak{h}_k$. Composing these isomorphisms with a complex Weyl group element that takes positive, non-compact imaginary (resp. real) roots of $\mathfrak{h}_l$ with respect to $\mathfrak{g}$ to the positive, non-compact imaginary (resp. real) roots of $\mathfrak{h}_l^J$, we get a candidate for $w_J$. We will fix such a candidate for each $J$ from now on.\\
Now, we have the formula
$$\widehat{\mathcal{O}_{\lambda}^M}|_{C_{l,J}(e)_M}=\frac{2^{l-k}\sum_{w\in W_{\text{imag.}}^M(\mathfrak{h}^J_l)}\epsilon(w)e^{iw\lambda}}{\pi^J_M}$$
for every $J$. Here $\pi_M^J$ is the product of the positive roots of $\mathfrak{h}_l^J$ with respect to $\mathfrak{m}$, and $W_{\text{imag.}}^M(\mathfrak{h}^J_l)$ is the Weyl group of the imaginary roots of $\mathfrak{h}^J_l$ with respect to $\mathfrak{m}$. We define
$$C_{l,J}(e)_M=\{X\in (\mathfrak{h}_l^J)'|\alpha(X)>0\ \forall\ \text{real\ roots}\ \alpha\in \Delta_M(\mathfrak{h}^J_l)\cap \Delta_G^+(\mathfrak{h}_l)\}$$
and more generally $$C_{l,J}(u)_M=u\cdot C_{l,J}(e).$$
This formula is proved in the same way as the special case of $\mathfrak{h}_l^J=\mathfrak{h}_l$, which is proved above.\\
\indent Partition $W_{k,l}=\bigsqcup_J W_{k,l}^J$ where $w\in W_{k,l}$ is in $W_{k,l}^J$ if $w\cdot\mathfrak{h}_l=\mathfrak{h}_l^J$. Every coset in $W^M_{\mathbb{R}}(\mathfrak{h}_l^J)\backslash W^G_{\mathbb{R}}(\mathfrak{h}_l^J)$ contains a unique representative $u$ such that $u^{-1}$ takes the positive real roots of $\mathfrak{h}_l^J$ with respect to $\mathfrak{m}$ to positive roots of $\mathfrak{h}_l^J$ with respect to $\mathfrak{g}$ and $u^{-1}$ fixes the imaginary roots of $\mathfrak{h}_l^J$ with respect to $\mathfrak{m}$. When we sum over $W^M_{\mathbb{R}}\backslash W^G_{\mathbb{R}}$, we will really be summing over this set of representatives. 
Then 
$$\sum_{u\in W^M_{\mathbb{R}}(\mathfrak{h}_l^J)\backslash W^G_{\mathbb{R}}(\mathfrak{h}_l^J)}\widehat{\mathcal{O}_{\lambda}^M}(uw_JX)|\pi_{G/M}^J(uw_JX)|^{-1}=$$
$$=2^{l-k}\sum_{u\in W^M_{\mathbb{R}}(\mathfrak{h}_l^J)\backslash W^G_{\mathbb{R}}(\mathfrak{h}_l^J)}\sum_{w\in W_{\text{imag.}}^M(\mathfrak{h}^J_l)}\frac{\epsilon(w)e^{iw\lambda(uw_JX)}}{\pi^J_M(uw_JX)|\pi^J_{G/M}(uw_JX)|}$$
$$=2^{l-k}\sum_{w\in W_{k,l}^J}\frac{\epsilon_{k,l}(w)e^{iw\lambda(X)}}{\pi_G(X)}.$$
\noindent The last equality follows from noticing that $\{w_J^{-1}u^{-1}w\}$ is really $W_{k,l}^J$ if $w$ varies over $W_{\text{imag.}}^M(\mathfrak{h}_l^J)$ and $u$ varies over our chosen set of representatives of $W_{\mathbb{R}}^M(\mathfrak{h}_l^J)\backslash W_{\mathbb{R}}^G(\mathfrak{h}_l^J)$. Further, we used  
$$\epsilon(w)=\epsilon_{k,l}(w_J^{-1}u^{-1}w)\ \text{and}\ \pi^J_M(uw_JX)|\pi^J_{G/M}(uw_JX)|=\pi_G(X).$$\\
\noindent Summing over all possible $J$, we get 
$$\widehat{\mathcal{O}_{\lambda}}|_{C_l(e)}=2^{l-k}\sum_J \sum_{w\in W_{k,l}^J}\frac{\epsilon_{k,l}(w)e^{iw\lambda}}{\pi}=\frac{2^{l-k}\sum_{w\in W_{k,l}}\epsilon_{k,l}(w)e^{iw\lambda}}{\pi}.$$
The vanishing of $\widehat{\mathcal{O}_{\lambda}}$ on the other Cartans follows from Harish-Chandra descent.
\end{proof}

\section{Fourier Transforms of Nilpotent Coadjoint Orbits for $\text{GL}(n,\mathbb{R})$}

Let $G=\text{GL}(n,\mathbb{R})$, let $\mathfrak{g}=\mathfrak{gl}(n,\mathbb{R})$, and let $\mathcal{O}_{\mathcal{L}}$ be as in Proposition 4.7. Let $\mathfrak{h}\subset \mathfrak{g}$ be a Cartan subalgebra, and let $H$ be the corresponding Cartan subgroup.
Put
$$\mathfrak{h}''=\{X\in \mathfrak{h}|\ \alpha(X)\neq 0\ \forall\ \text{real\ roots}\ \alpha\},$$
suppose $C\subset \mathfrak{h}''$ is a connected component, and put $C'=C\cap \mathfrak{h}'$. Choose positive roots of $(\mathfrak{g}_{\mathbb{C}},\mathfrak{h}_{\mathbb{C}})$ satisfying:\\
(i) $\alpha(X)>0$ for all positive real roots $\alpha$ and all $X\in C$\\
(ii) If $\alpha$ is a complex root, then $\alpha$ is positive iff $\overline{\alpha}$ is positive.\\
\\
Suppose $L\in \mathcal{L}$ with $L\supset H$. Let $W(G,H)_L$ be the stabilizer of $L$ in the real Weyl group of $G$ with respect to $H$, and let $W(L,H)$ be the real Weyl group of $L$ with respect to $H$. Note that the cardinality of the quotient
$$\left|\frac{W(G,H)_L}{W(L,H)}\right|$$ is independent of the choice of $L\in \mathcal{L}$ with $L\supset H$. Thus, we will denote the cardinality of this quotient by
$$\left|\frac{W(G,H)_{\mathcal{L}}}{W(\mathcal{L},H)}\right|.$$ 

\begin{theorem} We have the formula $$\widehat{\mathcal{O}_{\mathcal{L}}}|_{C'}=\left|\frac{W(G,H)_{\mathcal{L}}}{W(\mathcal{L},H)}\right|\sum_{L\supset H,\ L\in \mathcal{L}}\frac{\pi_L}{\pi}.$$
\end{theorem}

\begin{proof} By Proposition 4.7, we know 
$$\lim_{t\rightarrow 0^+} \mathcal{O}^G_{t\xi}=\mathcal{O}_{\mathcal{L}}\ \text{or}\ 0.$$
for any $\xi\in \mathfrak{g}^*$ such that $Z_G(\xi)\in \mathcal{L}$.\\
\noindent Fix such a $\xi\in \mathfrak{g}^*$ such that 
$$L=Z_G(\xi)\cong \text{GL}(q_1,\mathbb{R})\times \cdots \times \text{GL}(q_r,\mathbb{R}).$$ Then we can choose a fundamental Cartan and a labelling of roots in the last section so that $\mathfrak{h}_k\subset \mathfrak{l}$ is a fundamental Cartan and $\text{Lie}(H)=\mathfrak{h}_l$ with $l\geq k$. Note $$k=\sum \left\lfloor \frac{q_i}{2} \right\rfloor.$$ 
Choose positive roots of $\mathfrak{g}$ with respect to $\mathfrak{h}_l$ which satisfy the conditions (i) and (ii) stated at the beginning of this section. This determines positive roots for $\mathfrak{g}$ with respect to $\mathfrak{h}_i$ for every $i$. Put
$$(C^*)'=\{\lambda\in (\mathfrak{h}^*)'|\ \langle i\lambda, \alpha^{\vee} \rangle<0\ \text{for\ all}\ \alpha\in \Delta^+_L\}.$$
\noindent Then by Theorem 3.1,
$$\lim_{\lambda\rightarrow \xi,\ \lambda\in (C^*)'} \partial(\pi_L)|_{\lambda}\mathcal{O}^G_{\lambda}=i^{r(L,H)}|W(L,H)|\mathcal{O}^G_{\xi}.$$
\noindent And by Proposition 6.1,
$$\widehat{\mathcal{O}^G_{\lambda}}|_{C_l'(e)}=\frac{2^{l-k}\sum_{w\in W_{k,l}} \epsilon_{k,l}(w) e^{iw\lambda}}{\pi}.$$
\noindent Plugging this into the previous formula, we get
$$\widehat{\mathcal{O}^G_{\xi}}|_{C_l'(e)}=\frac{2^{l-k}}{|W(L,H)|}\frac{\sum_{w\in W_{k,l}}\epsilon_{k,l}(w) (w\pi_L)e^{iw\xi}}{\pi}.$$
\noindent Then $$\lim_{t\rightarrow 0^+}\widehat{\mathcal{O}_{t\xi}}|_{C_l'(e)}=\frac{2^{l-k}}{|W(L,H)|}\frac{\sum_{w\in W_{k,l}}\epsilon_{k,l}(w) (w\pi_L)}{\pi}.$$
\indent Now, using parts (i) and (ii) of the definition of $W_{k,l}$, we deduce that if $w\in W_{k,l}$, then $w\pi_L=\pm \pi_{L'}$ for some Levi $L'\in \mathcal{L}$ with $L'\supset H$. Conversely, one can deduce that whenever $L'\in \mathcal{L}$ with $L'\supset H$, there exists $w\in W_{k,l}$ such that $w\pi_L=\pm \pi_{L'}$ from the definition of $W_{k,l}$. In fact, using part (iii) of the definition of $W_{k,l}$ together with condition (ii) for our choice of positive roots in Theorem 7.1, we see that $w\pi_L=\epsilon_{k,l}(w)\pi_{L'}$.\\
\indent Combining these considerations, we get
$$\widehat{\mathcal{O}_{\mathcal{L}}}|_{C_l'(e)}=\frac{2^{l-k}}{|W(L,H)|}\sum_{L'\in \mathcal{L},\ L'\supset H}\#\{w\in W_{k,l}|\ wL=L'\}\frac{\pi_{L'}}{\pi}.$$
Finally, we use part (i) of the condition on our choice of positive roots, given at the beginning of this section together with the definitions of $C'$ and $C_l'(e)$ to realize $C'=C_l'(e)$. And a simple counting argument shows $$\frac{2^{l-k}}{\#\{w\in W_{k,l}|\ wL=L'\}}=|W(G,H)_L|.$$ The theorem follows.

\end{proof}

\section{A Corollary and a Conjecture}

An immediate corollary of our formula for $\widehat{\mathcal{O}_{\mathcal{L}}}$ is that the support of $\widehat{\mathcal{O}_{\mathcal{L}}}$ contains every Cartan $H\subset \mathcal{L}$. This follows from the fact that the numerator in our formula is a positive linear combination of products of positive roots; hence, it is a non-zero analytic function. Thus, we get the following corollary.

\begin{corollary} Let $\mathcal{O}_{\mathcal{L}}$ be the nilpotent coadjoint orbit associated to the conjugacy class of Levi factors $\mathcal{L}$. Then 
$$\operatorname{supp}(\widehat{\mathcal{O}_{\mathcal{L}}})=\overline{\bigcup_{L\in \mathcal{L}} \operatorname{Lie}(L)}.$$
\end{corollary}

In general, if $\pi$ is an irreducible, admissible representation of a reductive Lie group $G$, it is an interesting problem to determine the support its character, $\Theta_{\pi}$. As noted in the introduction, in the case $G=\text{GL}(n,\mathbb{R})$, the leading term of $\Theta_{\pi}$ at the identity is always a positive multiple of a the Fourier transform of a nilpotent coadjoint orbit \cite{BV}. Thus, this corollary gives the support of the leading term of an irreducible $\text{GL}(n,\mathbb{R})$ character.\\
\\
\indent Another observation about Theorem 7.1 is that the numerator is an integer linear combination of products of roots. Using results of \cite{V}, it was proven in \cite{SV} that the leading term of an irreducible character $\Theta_{\pi}$ is a nonnegative linear combination of Fourier transforms of nilpotent coadjoint orbits. The author thinks that it is an interesting question to ask whether the Fourier transforms of nilpotent coadjoint orbits also satisfy an integrality condition. Here is the author's conjecture.

\begin{conjecture} Suppose $\mathcal{O}\subset \mathfrak{g}^*$ is a nilpotent coadjoint orbit for a reductive Lie group $G$, suppose $H\subset G$ is a Cartan subgroup, and suppose $C\subset \mathfrak{h}'=\operatorname{Lie}(H)'$ is a connected component of the regular set. Let $\Delta$ be the roots of $\mathfrak{g}_{\mathbb{C}}$ with respect to $\mathfrak{h}_{\mathbb{C}}$, let $\Delta^+$ be a choice of positive roots, and let $\pi$ denote the product of these positive roots. Then
$$\pi\widehat{\mathcal{O}}|_C\in \mathbb{Z}[\Delta].$$
\end{conjecture} 

Of course, this conjecture would imply the analogous fact for leading terms of irreducible characters.

\section{Acknowledgments}
The author wishes to thank his advisor David Vogan for many conversations. The author would also like to thank Wulf Rossmann and Dan Barbasch for a few helpful tips and clarifications. 

\bibliographystyle{amsplain}

\end{document}